\documentclass[12pt,a4paper,reqno]{amsart}

\usepackage{amsmath,amsfonts,amsthm,amssymb,graphicx}

\usepackage[shortlabels]{enumitem}

\usepackage[dvipsnames]{xcolor}
\usepackage{hyperref}

\usepackage{orcidlink}

\usepackage[marginratio=1:1,totalwidth=15.75cm,totalheight=22.275cm]{geometry}

\newcommand{\Area}{\operatorname{Area}}

\usepackage{subcaption}

\usepackage{graphicx}

\theoremstyle{plain}
\newtheorem{thm}{Theorem}[section]
\newtheorem{lem}[thm]{Lemma}

\newtheorem{question}[thm]{Question}

\newcommand{\dA}{\mathrm dA}
\newcommand{\Deriv}{\mathrm D}

\renewcommand{\subset}{\subseteq}
\renewcommand{\theta}{\vartheta}
\renewcommand{\phi}{\varphi}

\theoremstyle{definition}

\theoremstyle{remark}
\newtheorem{rmk}[thm]{Remark}

\newcommand{\eps}{\varepsilon}
\newcommand{\interior}{\operatorname{int}}

\newcommand{\DD}{\mathbb{D}}
\newcommand{\dist}{\operatorname{dist}}

\numberwithin{equation}{section}
\newcommand \C{\mathbb{C}}
\newcommand \Ch{\widehat{\mathbb{C}}}

\newcommand \R{\mathbb{R}}
\newcommand \D{\mathbb{D}}

\newcommand*{\defeq}{\mathrel{\vcenter{\baselineskip0.5ex \lineskiplimit0pt
                     \hbox{\scriptsize.}\hbox{\scriptsize.}}}%
                     =}
\newcommand{\eqdef}{=\mathrel{\vcenter{\baselineskip0.5ex \lineskiplimit0pt
                     \hbox{\scriptsize.}\hbox{\scriptsize.}}}}

\begin{document}

\title{Absence of bounded-orbit wandering domains}
\author[N. Prochorov]{Nikolai Prochorov\orcidlink{0000-0001-9725-299X}}

\address{\noindent Dept. of Mathematics \\ The University of Manchester \\ Manchester \\ M13 9PL \\ UK}
\email{nikolai.prochorov@manchester.ac.uk}

\author[L. Rempe]{Lasse Rempe\orcidlink{0000-0001-8032-8580}} 
\address{\noindent Dept. of Mathematics \\ The University of Manchester \\ Manchester \\ M13 9PL \\ UK}
\email{lasse.rempe@manchester.ac.uk}

\author[J. Waterman]{James Waterman\orcidlink{0000-0001-7266-0292}}
\address{Center for Computing Sciences\\ Bowie \\ MD 20715 \\ USA}
\email{jawater@super.org}


\date{\today}

\begin{abstract}
 We prove that transcendental entire functions do not have bounded-orbit wandering domains. This
  answers a long-standing open question. In fact,
  we prove a more general theorem, establishing the absence of certain wandering domains even for
  locally defined analytic functions. 


 The proof is based on a recent remarkable argument of Ye, who obtained an elementary proof of
  Sullivan's classical theorem that rational maps do not have wandering domains. 
\end{abstract}

\maketitle

\section{Introduction}
Let $f\colon X\to X$ 
be either a rational function
 $f\colon\Ch\to\Ch$, or a 
 transcendental entire function
 $f\colon \C\to\C$. We may consider 
 $f$ as the transition rule of a
 discrete-time dynamical system. The
 state space $X$ then splits into two
 fundamental parts, 
 defined and studied independently
 by Fatou and Julia in the
 early 20th century in the case
 of rational functions, and extended 
 to transcendental entire functions
 by Fatou in 1926~\cite{fatou26}. These
 sets are now called the \emph{Fatou set} 
 and the \emph{Julia set}. The former
 is an open set and 
 consists of those points where the 
 dynamics of $f$ is \emph{stable} under
 perturbation. More formally, 
 the iterates $\{f^n\}_{n=0}^{\infty}$ 
 are equicontinuous (with respect to the
 spherical metric) near points of $F(f)$.
 The complement is the 
 \emph{Julia set} $J(f)$, where the dynamics is
 unstable.
 We refer to~\cite{milnor06} for background
  in holomorphic dynamics, particularly
  the iteration of rational functions,
  and to~\cite{bergweiler93} for a 
  survey on the iteration of
  transcendental entire and meromorphic functions.

 The connected components of $F(f)$ are 
 called \emph{Fatou components}.
 If $U$ is  a Fatou component, then
 for every $n\geq 0$, the $n$-th 
 forward image
 $f^n(U)$ is contained in a Fatou component
 $U_n$. If $U_n\neq U_m$ for all $n\neq m$, 
 then $U$ is called \emph{wandering};
 otherwise, $U$ is called \emph{eventually periodic}. 
 Eventually periodic Fatou components
 can be classified into  a small 
 finite number of different types, each of whose dynamics
 is well-understood. 

Entire functions may have wandering domains,
  while rational maps do not by a famous theorem of Sullivan~\cite{sullivan85}. 
It is known that every 
 wandering domain $U$ must belong to one
 of three types:
 \begin{enumerate}
    \item in an \emph{escaping} 
      wandering domain, $\lvert f^n(z)\rvert \to \infty$ for every $z\in U$;
     \item in an \emph{oscillating}
      wandering domain, 
        $\liminf_{n\to\infty} \lvert f^n(z)\rvert < \infty =\limsup_{n\to\infty} \lvert f^n(z)\rvert$
        for every $z\in U$;
     \item in a \emph{bounded-orbit}
     wandering domain, 
     $\limsup_{n\to\infty} \lvert f^n(z)\rvert < \infty$ for every $z\in U$.
 \end{enumerate}

 The first example of a wandering domain,
  constructed by Baker~\cite{baker76},
  was an escaping wandering domain.
  Eremenko and Lyubich~\cite{eremenko-lyubich87} constructed the first
  known examples of oscillating
  wandering domains. Wandering domains have been a highly active area of research recently;
  see for example~\cite{befrs,bocthaler21,evdoridou-glucksam-pardosimon23,pardosimon-sixsmith24,martipete-rempe-waterman25,evdoridou-marti-pete-rempe26} and the references in these papers. The existence of bounded-orbit
  wandering domains, on the other hand,
  is a long-standing open question, 
  versions of which appeared three times on Hayman's
  list of \emph{Research Problems in Complex Analysis}~\cite{hayman-lingham19}, in 1984 and 1989 (twice).
  They are listed as Problems 2.67, 2.77 and~2.87 and attributed to Baker, Eremenko, Herman, Kra and Lyubich.

\begin{question}[Bounded-orbit wandering domains]\label{question:bounded-orbit}
 Does there exist a transcendental entire function with a bounded-orbit wandering domain?
\end{question}

 Orbits in hypothetical wandering domains of rational functions would automatically be ``bounded'' in the sense that they are compactly contained in the domain of definition (which is itself compact). 
 Moreover, on a compact subset of the complex plane, there is no discernible difference between a polynomial,
 a transcendental entire
 function, and a general holomorphic function. Indeed, any 
 holomorphic function on a neighbourhood of a compact
 set that does not separate the plane can be arbitrarily
 closely approximated 
 by a polynomial, by Runge's theorem~\cite{runge85}. So 
 it is natural to ask whether different bounded-orbit
 dynamics can arise for
 transcendental entire functions than for polynomials,
 or not. However, Sullivan's proof of 
 the no wandering domains theorem uses deformation
 theory and 
  relies crucially on the fact that rational maps of a given degree depend only on finitely
  many parameters. So it remained unclear whether one should expect the existence of
   bounded-orbit wandering domains for entire functions.

 Recently, a remarkable new proof of Sullivan's theorem was proposed by Ye~\cite{ye26}. This proof
  is essentially elementary, using only well-known properties of the hyperbolic metric
  in plane domains and the Gauss-Bonnet theorem. This is a truly remarkable development,
  given how long Sullivan's proof stood essentially without improvement~-- indeed, in 2014,
  the second author expressed strong scepticism that a simple  proof of
  the no wandering domains theorem could be found~\cite{rempe-mathoverflow-2014}. The proof in~\cite{ye26} also applies to transcendental entire 
  functions with finitely many critical and asymptotic values,
  a result previously proved in~\cite{goldberg-keen86} and~\cite{eremenko-lyubich92}
  using Sullivan's techniques, and indeed to a more
  general class for which the theorem had not previously been
  established: those for which the derived set of the
  set of singular values is contained in the Fatou set. 

  In this paper, we further develop the ideas pioneered in~\cite{ye26} to resolve
  Question~\ref{question:bounded-orbit} in the negative. 
  
\begin{thm}[No bounded-orbit wandering domains]\label{thm:boundedorbitentire}
 Let $f$ be a transcendental entire function, and let $U$ be a wandering Fatou component 
   of $f$. Then there exists a strictly increasing sequence $(n_k)_{k=0}^{\infty}$ 
   such that $f^{n_k}|_U\to\infty$ locally uniformly as $k\to\infty$. 
\end{thm}

In fact, we prove a more general statement for locally
defined mappings, as follows. 
Suppose that $f\colon O\to\Ch$ is an open meromorphic function, where $O\subset\Ch$ is open.
Given an open set $V\subset O$, we define the set 
\begin{equation}\label{eqn:TV}
 T_V(f)\defeq \{z\in V:f^n(z)\in V\text{ for all }n\ge0\}
\end{equation}
 of points that stay in $V$ under iteration of $f$. Its interior 
  \begin{equation}\label{eqn:OmegaV}
  \Omega_V(f) \defeq\interior(T_V)\end{equation} 
 consists of all points that have an entire neighbourhood that
  remains in $V$ under iteration. Let $U$ be a connected component of $\Omega_V(f)$. We say that 
  $U$ is \emph{wandering} if $f^n(U)$ and $f^m(U)$ are contained in different 
  connected components of $\Omega_V(f)$ for all $n\neq m$. 
\begin{thm}[Local formulation on the sphere]\label{thm:spherelocal}
Let $O\subset\Ch$ be open and $f:O\to\Ch$ open and holomorphic.
Let $V\Subset O$ be open with
$\Ch\setminus\overline V\ne\varnothing$. 
\begin{enumerate}[(1)]
\item\label{it: no wandering orbit} Let $U$ be a wandering component of $\Omega_V(f)$, and let $z\in U$. Suppose that the connected component $U_n$ of $\Omega_V(f)$ containing $f^n(U)$ is simply connected for every $n\geq 0$. Then there is a strictly increasing sequence $(n_k)_{k=0}^{\infty}$ such that $\dist^{\#}(f^{n_k}(z),\partial V)\to 0$, where $\dist^{\#}$ denotes
spherical distance in $\Ch$.
\item\label{it: no wandering set} Suppose that $A\subset T_V(f)\setminus\Omega_V(f)$ is measurable such that
 \begin{enumerate}[(a)]
  \item $f^n(A)\cap f^m(A) =\varnothing$ for $n\neq m$,
   \item $f$ is injective on $\bigcup_{n=0}^{\infty} f^n(A)$, and
   \item $A$ has positive area. 
  \end{enumerate}
  Then $\bigcup_{n=0}^{\infty} f^n(A)$ is not contained in any compact subset of $V$.
\end{enumerate}
\end{thm}

The idea of the proof in~\cite{ye26} relies on the fact that removing $N$ points from 
 a finitely punctured sphere increases the total hyperbolic area by $2\pi N$
 (compare Lemma~\ref{lem:increment}
 below). This, together with the fact that $f$ acts as a covering map over the complement
 of the set of 
 singular values, is used to conclude that the hyperbolic area of the wandering domain
 in question is finite, which is a contradiction to the fact that the wandering domain
 is conformally equivalent to the hyperbolic plane, which has infinite area.

 In our setting, we \emph{cannot} exclude the 
 singular values from our surfaces under consideration. The idea
 for dealing with this issue is as follows. If our
 wandering domain remains within a bounded part of the complex plane, then it is natural 
 to remove only those critical values whose critical points likewise lie within this
 bounded part. The function is not a covering map when considering the preimage of the surface,
 but it behaves similarly to a covering on the part of the Riemann surface of the inverse function
 (thought of as spread over the plane) that corresponds to our bounded domain. The authors
 explored this idea in collaboration with generative AI, and
 arrived at the first version of the proof of Theorem~\ref{thm:boundedorbitentire}. The proof
 presented here is fully written by the authors,
 and hopefully has added clarity and context
 over the version that first emerged from
 discussions with AI tools. The proof has
 also been auto-formalised in the Lean 
 proof assistant using generative AI~\cite{PRW2026formalisation}.

We note that we have formulated our proofs for 
 subsets of the Riemann sphere, but they should 
 carry over to functions $O\to \Sigma$, where
 $\Sigma$ is a Riemann surface and $O\subset \Sigma$
 is open. Using similar ideas, we have also been able to
 prove the 
 following result, which answers another open
 question in transcendental dynamics: 
 any wandering domain of a transcendental entire function
 $f$ has a limit function that belongs to the derived set of
 singular values of $f$. The proof will be
 provided in a subsequent revision of the article. 

\subsection*{Authors' statement on the use of generative AI}
 As mentioned above, generative AI (ChatGPT 6 Astra, 5.6 Sol, and Claude Opus 5) was used in a substantial way to develop the first draft of
  the proof of 
  our main theorem, starting from our understanding of the proof in~\cite{ye26} and our ideas
  for extending it. The proof was then presented and written by the authors in their 
  own words.
 Generative AI (ChatGPT 6 Astra) was
 also used to auto-formalise the proof
 in the Lean 
 proof assistant, and to proof-read
 a completed draft of the paper.

\subsection*{Acknowledgements}
 Our rapid development of the ideas that led to the proof of Theorem~\ref{thm:boundedorbitentire} 
  would not have been possible without the assistance of generative AI, nor without the 
  key ideas in~\cite{ye26}, which according to the author's statement was also assisted by
  AI. In turn, this assistance would undoubtedly not have been possible without the human mathematicians
  and thinkers who came before, and whose work is the basis on which modern AI tools
  have been built. It is our hope that the benefits of today's powerful new AI tools will
  likewise be used to benefit mathematics, and humanity, as a whole.

\section{Preliminaries} 

\subsection*{Basic notation}
As usual, we define the complex plane by $\C$ and
the Riemann sphere by $\Ch$. Boundaries and 
closures are taken in $\C$ if not explicitly
stated otherwise. We write 
 $V\Subset U$ to mean that
 $V$ is \emph{compactly contained} in $U$; that is,
 the closure of $V$ in $U$ is a compact subset of $U$.  
Euclidean distance is denoted by $\dist$, and
Euclidean balls with centre $z$ and radius $r$ 
are denoted $B(z,r)$.
If $x\in\R$, we set $x_+\defeq \max(x,0)$.

\subsection*{Hyperbolic geometry}
Suppose that $X$ is a hyperbolic Riemann surface~-- that is,
  one which is not
 a torus, a sphere or a once- or twice-punctured sphere.
 Then $X$ carries a unique complete conformal  metric 
 of constant curvature $-1$, called the 
 \emph{hyperbolic metric}. In local coordinates,
 we may write this metric as $\rho_X(z)\lvert dz\rvert$,
 where $\rho_X$ is a real-analytic function of $z$.
 In particular, if $X\subset\C$, then we may 
 describe the metric globally by its \emph{density} $\rho_X$
 with respect to the Euclidean metric. 
 The condition that the metric has curvature $-1$ means
 that $\Delta \log \rho_X = \rho_X^2$. 
 Hyperbolic distance in $X$ is denoted
 $\dist_X$, and hyperbolic discs are
 denoted $B_X(z,r)$.

 The density of the hyperbolic metric on the unit disc 
  is given by 
    \[ \rho_{\DD}(z) = \frac{2}{1-\lvert z\rvert^2}. \]
 A key fact is that holomorphic maps between
 hyperbolic surfaces do not 
  expand hyperbolic metrics; see \cite[Theorem 2.11]{milnor06}.

 \begin{thm}[Pick's theorem]
   Let $X$ and $Y$ be hyperbolic surfaces, and let
   $f\colon X\to Y$ be holomorphic. Then 
   the \emph{hyperbolic distortion} of $f$ at a point
   $z\in X$ is the quantity
   \[ \| \Deriv f(z)\|_X^Y \defeq \lvert f'(z)\rvert \cdot 
       \frac{\rho_Y(f(z))}{\rho_X(z)}\]
     (where the densities are understood in local 
     coordinates). 
     
     This quantity satisfies
      \[\| \Deriv f(z)\|_X^Y \leq 1\]
      for all $z\in X$, with equality if and only if
      $f\colon X\to Y$ is a holomorphic covering map. 

    In particular, if $X\subset Y$, then
     $\rho_X \geq \rho_Y$.
 \end{thm}

We also make use of the following version of a standard
hyperbolic density estimate; see~\cite[Proposition 3.4]{mihaljevicbrandt-rempegillen13} or~\cite[Theorem 1]{minda17}.
\begin{lem}\label{lem:hypercompare} There
exist strictly decreasing, surjective
and continuous functions
\[\alpha,\beta \colon (0,\infty)\to (1,\infty)\]
with the following property.
Let $X$ be a hyperbolic Riemann surface and let $Y\subsetneq X$ be open and $z\in Y$. Then
\[ 1\leq \alpha(\delta) \leq \frac{\rho_Y(z)}{\rho_X(z)} 
  \leq \beta(\delta),\]
 where $\delta = dist_X(z,X\setminus Y)$, and the
 quotient of densities should be understood in
 local coordinates. 
\end{lem}
\begin{rmk}
 The exact expressions 
 for $\alpha$ and $\beta$ are given in~\cite[Proposition 3.4]{mihaljevicbrandt-rempegillen13}, but we do not require
 them. 
\end{rmk}

 In our proofs, we will approximate 
  hyperbolic domains of the plane by finitely punctured
  spheres. The relation between the 
  hyperbolic metric of the limit to that of its
  approximations is
  described by the 
  following qualitative lemma. (See also~\cite[Lemma 2.2]{ye26}.)
\begin{lem}\label{lem:convergence of densities}
 Let $J\subset\C$ be closed and non-empty, and let 
   $(P_j)_{j=0}^{\infty}$ be an increasing sequence of
   finite subsets with $\# P_j \geq 2$, and such that
   $P_\infty \defeq \bigcup_{j=0}^{\infty} P_j$ 
   is dense in $J$. Then: 
    \begin{enumerate}
        \item\label{it:boundary} If $z \in J \setminus P_\infty$, then $\rho_{\C \setminus P_j}(z) \to \infty$ as $j \to \infty$.
        \item\label{it:interior} If $z$ belongs to a connected component $U$ of $\C \setminus J$, then $\rho_{\C \setminus P_j}(z)$ converges locally uniformly on $U$ to $\rho_U(z)$ as $j \to \infty$.
    \end{enumerate}
\end{lem}

\begin{proof}
    Let $z \in J \setminus P_\infty$. 
     Then 
     $\lim_{j\to\infty}\dist(z, P_j)=0$ by choice of
     the sets $P_j$. The hyperbolic metric of $\C\setminus P_0$ and the Euclidean metric are comparable in
     a neighbourhood of $z$, so 
     also $\lim_{j\to\infty} \dist_{\C\setminus P_0}(z,P_j)=0$. By Lemma~\ref{lem:hypercompare},
     $\rho_{\C\setminus P_j}(z)/\rho_{\C\setminus P_0}(z)\to\infty$, establishing~\eqref{it:boundary}. 

    To prove~\ref{it:interior}, by Lemma~\ref{lem:hypercompare} it is enough to
    show that 
     \[ \lim_{j\to\infty}\dist_{\C\setminus P_j}(z,\zeta)\to\infty \] for every $z\in U$ and every
     $\zeta\in\partial U\setminus P_{\infty}$.
     To see this, let $\eps>0$ and let $n$ be 
     so large that there are $p_1,p_2\in P_n$ with
     $\dist(\zeta,p_2)<\dist(p_1,p_2)<\eps$. Let
     $\lambda\colon \C\to\C$ be the affine map that
     takes $p_1$ to $0$ and $p_2$ to $1$; then
     $\lambda(\zeta)$ remains bounded as $\eps\to 0$, while
     $\lambda(z)\to\infty$. We conclude that
     \[ \dist_{\C\setminus P_j}(z,\zeta) \geq 
         \dist_{\C\setminus \{p_1,p_2\}}(z,\zeta) =
         \dist_{\C\setminus\{0,1\}}(\lambda(z),\lambda(\zeta)) \to\infty \] 
         as $\eps\to 0$. 
\end{proof}

\begin{lem}\label{lem:increment}
    Let $P, Q \subset \widehat{\C}$ be two finite subsets such that $P \subset Q$ and $|P| \geq 3$. Then for any measurable subset $U \subset \widehat{\C} \setminus Q$, 
    $$
        \Area_{\widehat{\C} \setminus Q}(U) - \Area_{\widehat{\C} \setminus P}(U) \leq 2\pi(|Q| - |P|).
    $$
\end{lem}
\begin{proof}
  It is a classical consequence of the Gauss--Bonnet
  theorem that the hyperbolic area of 
    an $n$-times punctured sphere is 
    $2\pi (n-2)$; see~\cite[Theorem~10.4.3]{Beardon1983}. The claim follows
    from Pick's theorem.
\end{proof}


\subsection*{Subharmonic functions}
We use the following standard facts concerning subharmonic  functions. Suppose $D\subset \C$ is a domain, then a function $u\colon D \to [-\infty,\infty)$ is subharmonic if it is upper semicontinuous and satisfies the mean value inequality
\[u(z) \leq \frac{1}{2\pi}\int_0^{2\pi} u(z+re^{it})dt,\]
for each $\overline{B(z,r)}\subset D$. If $u\in C^2(D)$, then this is equivalent to $\Delta u \geq 0$, where $\Delta$ denotes the Laplacian. 

If $u$ is subharmonic, 
the distributional Laplacian $\Delta u$ is positive in the sense of distributions. Hence by the Riesz representation theorem there exists a unique positive measure $\mu_u$ on $D$ such that
\[\int_D \chi d\mu_u=\int_D u\Delta \chi \dA\] for every $\chi\in C^\infty_c(D)$. We call $\mu_u$ the Riesz measure of $u$. See \cite{hayman-kennedy76}.

\section{Uniform area estimates on compact sets}\label{sec:loss}
In this section, we prove the key result for our proof, which concerns a comparison
of hyperbolic areas. In order to explain the idea, we briefly review Ye's proof of 
Sullivan's theorem. Suppose that $f$ is a rational function. We may choose some
finite forward-invariant approximation $P$ of the Julia set; that is, 
$P\subset J(f)$ and every point of $J(f)$ is within $\eps>0$ of some point in $P$. 
Consider the finitely punctured spheres $X\defeq \Ch\setminus P$ and
$Y\defeq X\setminus S(f)$, where $S(f)$ is the set of critical values of $f$. 
We have $f^{-1}(Y)\subset X$ and $f\colon f^{-1}(Y)\to Y$ is a covering map.
So by Pick's theorem, the hyperbolic metric on $f^{-1}(Y)$, which is precisely the pull-back 
of the hyperbolic metric on $Y$ under $f$, is no smaller than that on $X$. 

Now suppose that $(U_n)_{n=0}^{\infty}$ is an orbit of simply connected wandering domains
 such that the map $f$ is univalent along the orbit $U_n$. (A well-known argument shows that
 some forward image of a wandering domain of a rational function would have this property.) 
 Then the above argument shows that the hyperbolic area of the 
 wandering domains does not decrease under $f$, considered as a map from a subset of $X$ to
 $Y$, so 
   \[ \Area_{Y}\left(\bigcup_{j=1}^{\infty} U_n\right) \geq 
      \Area_{X} \left(\bigcup_{j=0}^{\infty} U_n\right). \]
 But only $U_0$ is missing from the upper union compared to the lower one, 
  and removing the finitely many points of $S(f)$ from $X$ increases the total
  hyperbolic area by at most $2\pi \# S(f)$. So we see that $\Area_X(U_0)$ is bounded by
  $2\pi \# S(f)$, which is a bound that is \emph{independent} of the size of $P$. On the other hand,
  if we choose $P$ sufficiently large, then the hyperbolic metric of $X$ on $U$ converges
  to the hyperbolic metric on $U$, which is a simply connected domain and therefore has infinite
  hyperbolic area. This yields the desired contradiction.

 The main difficulty that we need to overcome is that,
  for a general transcendental entire function,
  the set $S(f)$ of singular values 
  may be infinite; indeed, it may coincide with the entire complex plane. The key observation is
  the following. Form $X$ as before, but restrict our
  attention to a compact part $K$ of the complex plane. 
  Let $V$ be an open neighbourhood of $K$ and,
  in defining the surface $Y$, only remove those critical values that correspond to critical points in $V$. 
  Then $f\colon f^{-1}(Y)\to Y$
  is no longer a covering, so if
  we pass from the hyperbolic metric of $X$ to   the pull-back $\eta_P$ of the hyperbolic metric of $Y$,
  the metric may decrease in some places. The goal of this
  section is to show that the total decrease of such area is bounded by
  a constant, independently of $P$. 
  This is the statement of the following theorem, 
  which allows the remainder of the
  argument to go through similarly as before. 

\begin{thm}[Comparison of hyperbolic area and pullback]      \label{thm:localdeficit}
Let $V\subset\C$ be bounded and open, and let $f$ be open and meromorphic  on
a neighbourhood of $\overline V$. Fix a compact set
$K\subset V$ and distinct $a,b\in\C$. Then there is $C\geq 0$ with the following property. 

Define 
$E\defeq f(\{z\in V:f'(z)=0\})$.
For every finite set $P\supset\{a,b\}$ satisfying
$f(P\cap V)\subset P$, write
\[
 X_P\defeq \C\setminus P,\quad Y_P\defeq \C\setminus(P\cup E) \]
 and let $\eta_P$ be the density of the pullback of the hyperbolic metric of $Y_P$ 
 under $f$: 
 \[
 \eta_P(z)\defeq \lvert f'(z)\rvert \rho_{Y_P}(f(z)), \quad z \in f^{-1}(Y_P).
\]
Then 
\[
 \int_{K\cap f^{-1}(Y_P)}
       (\rho_{X_P}^2-\eta_P^2)_+\,\dA\le C.
\]
\end{thm}  

To prove Theorem~\ref{thm:localdeficit}, we first
 bound the ratio of the corresponding
 hyperbolic metrics by a uniform constant. 

\begin{lem}[Comparison of hyperbolic metric and pullback]\label{lem:localratio}
Under the assumptions of  Theorem~\ref{thm:localdeficit},
there is a constant $M\geq 1$, which is 
independent of $P$, such that
\[
 \frac{\rho_{X_P}(z)}{\eta_P(z)}\leq  M
\]
for all $z\in K\cap f^{-1}(Y_P)$. 
\end{lem}
\begin{proof}
Observe that $\infty\notin Y_P$, so $f$
has no poles in $f^{-1}(Y_P)$, so
the density $\eta_P(z)$ is indeed 
defined and finite.

Let $z_0\in K$ and choose neighbourhoods $G'\Subset G\Subset V$ of $z_0$ and a disc $D$ centred at $f(z_0)$ so that $f\colon G \rightarrow D$ is proper, and $f(\overline{G'})\Subset D$.
(If $z_0$ is a pole, then $D$ is chosen to be
a disc on the Riemann sphere centred 
at $\infty$, and the calculations below
should be understood to be carried out
in local coordinates.)

Now, note that the restriction $f\colon G \cap f^{-1}(Y_P) \to D \cap Y_P$ is a covering map and $G\cap f^{-1}(Y_P)\subset X_P$. Indeed, if $z \in G\cap f^{-1}(Y_P)$ were in $P$, then $f(z)\in P$  by $f(P\cap V)\subset P$, contradicting $f(z)\in Y_P$.
It follows by Pick's theorem that
    \[\rho_{X_P}(z) \leq \rho_{G\cap f^{-1}(Y_P)}(z)=|f'(z)|\rho_{D\cap Y_P}(f(z)).\]
    Hence,
    \[\frac{\rho_{X_P}(z)}{\eta_P(z)}\leq \frac{\rho_{D\cap Y_P}(f(z))}{\rho_{Y_P}(f(z))}, \]
    for $z\in G' \cap f^{-1}(Y_P)$.
    It remains to compare the hyperbolic metrics of $D\cap Y_P$ and $Y_P$ on $f(G')$.

Now, since $a,b\in P$, by the definition of $Y_P$ we have that $\rho_{Y_P}\geq \rho_{\C\setminus\{a,b\}}$. Moreover, the image of $\overline{G'}$ is compactly contained in $D$. The density $\rho_{\C\setminus\{a,b\}}$ is bounded from below on $D\setminus \{a,b\}$ and the Euclidean distance between $f(\overline{G'})$ and  $\partial D$ is positive. So there exists $\delta>0$, independent of $P$, such that $d_{Y_P}(z,Y_P\setminus D)\geq \delta$ for $z\in f(G')\cap Y_P$.
By Lemma~\ref{lem:hypercompare}, we obtain a constant $M(G')$, independent of $P$, so that
\[\frac{\rho_{D\cap Y_P}(f(z))}{\rho_{Y_P}(f(z))} \leq M(G'),\]
for $z\in f(G')\cap Y_P$.
Finally, finitely many neighbourhoods $G'$ cover $K$, so we may take the maximum of each of the corresponding 
finitely many constants $M(G')$ to obtain $M$.   
\end{proof}

The idea now is that, in the situation of Lemma~\ref{lem:localratio}, the 
  function $\log\left(\frac{\rho_{X_P}}{\eta_P}\right)_+$ is a
  subharmonic function on $V\cap f^{-1}(Y_P)$, which is locally bounded
  near each point of $V\setminus f^{-1}(Y_P)$. As a result, we will be 
  able to bound its integral on $K$ by a bound that is independent of $P$, using  
  the following fact about subharmonic functions. 

\newcommand{\supp}{\operatorname{supp}}

\begin{lem}\label{lem:positivecutoff}
Let $F$ be a finite subset of an open set $V\subset\C$, and let $u\in C^2(V\setminus F)$
be locally bounded above near $F$. Suppose that $u(z)>0$ if and only if
$\Delta u(z)>0$, for all $z\in V\setminus F$. 
Let $\chi\in C_c^{\infty}(V)$ be a non-negative test function, and suppose that there are  
   $m\geq 0$ and a neighbourhood $U$ of $\supp \Delta\chi\subset \supp \chi$ such that $u(z)\leq m$
   for all $z\in U\setminus F$. Then
\[
 \int_{V\setminus F}\chi \cdot(\Delta u)_+\,\dA
 \le m\int_V|\Delta\chi|\,\dA.
\]
\end{lem}
\begin{proof}
Set $v=u_+$ on $V\setminus F$. Then $v$ is subharmonic on $V\setminus F$, since on $\{z\colon u(z)>0\}$ we have $v=u$ and $\Delta u>0$. On the other hand, if $u(z)\leq 0$,
then $v(z)=0$, and $v$ satisfies the mean value inequality
at $z$, since $v$ is globally nonnegative. 

Now $u$ is locally bounded above near $F$, and hence so is $v$. Because $F$ is a finite (and so polar) set, \cite[Theorem 5.18]{hayman-kennedy76} implies that $v$ extends to a subharmonic function on $V$, with
$v(z)\leq m$ for $z\in U$. 

Consider the Riesz measure of $v$, which we denote $\mu_v$. On $\{z\colon u(z)>0\}$, we have $v=u$  by definition, so $\mu_v=\Delta u ~\dA$. Also, since $u(z)>0$ if and only if $\Delta u>0$, we have that $(\Delta u)_+=0$ on $\{z: u(z) \leq 0\}$. 
Hence,
\[\mu_v \geq 1_{V\setminus F}\cdot (\Delta u)_+ \dA.\]
Since $\chi$ is nonnegative,
\[\int_{V\setminus F}\chi\cdot (\Delta u)_+ \dA \leq \int_V\chi d\mu_v.\]
Now by the definition of the Riesz measure,
\[\int_{V}\chi d\mu_v = \int_Vv\Delta\chi \dA.\]
Finally, since $v\leq m$ on $U$, which is 
  a neighbourhood of the support of $\Delta\chi$,
\[\int_Vv\Delta\chi \dA \leq m \int_V|\Delta\chi| \dA .\]
Combining the inequalities, we obtain the desired result.
\end{proof}
\begin{proof}[Proof of Theorem~\ref{thm:localdeficit}]
    Choose a compact set $K'\subset V$ with $K\subset \interior(K')$ and choose  a nonnegative
    test function $\chi \in C_c^\infty(K')$ such that $\chi=1$ on a  neighbourhood of $K$. By Lemma~\ref{lem:localratio} applied to $K'$, there is $M\geq 1$, independent of $P$, such that 
    \[
 \frac{\rho_{X_P}(z)}{\eta_P(z)}\leq  M
\] for all $z\in K' \cap f^{-1}(Y_P)$.

Let $F\defeq V\setminus f^{-1}(Y_P)=V\cap f^{-1}(P\cup E\cup\{\infty\})$, which is a finite set, and define $u(z)\defeq \log \bigl({\rho_{X_P}(z)}/{\eta_P(z)}\bigr)$ on $V\setminus F$.
By Lemma~\ref{lem:localratio}, $u$ is locally bounded above near $F$.
In particular, $u \leq \log M$ on a neighbourhood of the support of $\Delta \chi$, since the support of $\Delta \chi$ is compactly contained in $K'$.

Since hyperbolic metrics have curvature $-1$, $\Delta \log \rho_{X_P}=\rho^2_{X_P}$, and similarly for
$\rho_{Y_P}$. Since curvature is a conformal invariant
 and $f'(z)\neq 0$ on $V\setminus F$, the pullback 
 $\eta_P$ of $\rho_{Y_P}$ under $f$ 
 also has constant curvature $-1$, i.e.\ 
$\Delta \log \eta_P = \eta_P^2$. 
(This can also be checked directly by expanding \[\Delta \log \eta_P=\Delta \log |f'| + \Delta(\log \rho_{Y_P} \circ f)=|f'|^2(\Delta \log \rho_{Y_P})\circ f,\] and then substituting $\Delta \log \rho_{Y_P}=\rho^2_{Y_P}$.)

Hence, our function $u$ satisfies $\Delta u= \rho^2_{X_P} - \eta_P^2$. Moreover, $u(z)>0$ if and only if $\rho_{X_P}(z)>\eta_P(z)$ if and only if $\Delta u(z) >0$.
We can thus apply Lemma~\ref{lem:positivecutoff} to  $u$, 
which gives
\[\int_{V\setminus F}
       \chi \cdot (\rho_{X_P}^2-\eta_P^2)_+\,\dA \le \log M\int_V|\Delta\chi|\,\dA,\]
and since $\chi=1$ on $K$,
\[
 \int_{K\cap f^{-1}(Y_P)}
       (\rho_{X_P}^2-\eta_P^2)_+\,\dA\le \log M\int_V|\Delta\chi|\,\dA \eqdef C\geq 0,
\]
as claimed. (Note that 
the right hand side is indeed independent of $P$.)
\end{proof}

\section{Absence of bounded-orbit wandering domains for locally defined functions}\label{sec:absence}

With Theorem~\ref{thm:localdeficit} in place, we 
are now ready to carry the central argument
of~\cite{ye26} over to our more general setting,
where it takes the following form. 
\begin{lem}[Uniform area bound]\label{lem:localcancel}
We continue to use the notation of Theorem~\ref{thm:localdeficit}, 
with its fixed compact set $K$ and constant
 $C$.
Let $W\subset K\cap f^{-1}(Y_P)$ be measurable and suppose that  $f|_W$ is injective.
If $B\subset W$ is measurable and $f(W)\subset W\setminus B$,
then
\[
 \Area_{X_P}(B)\le C+2\pi|E|.
\]
\end{lem}

\begin{proof}
    Theorem~\ref{thm:localdeficit} implies that
\begin{align*}
    \Area_{X_P}(W)
    &= \int_W \rho_{X_P}^2\,\dA \leq \int_W \eta_{P}^2\,\dA + \int_W (\rho_{X_P}^2 - \eta_P^2)_+\,\dA  
    \leq \int_W \eta_P^2\,\dA + C \\
    &= \int_W |f'(z)|^2 \rho_{Y_P}^2(f(z))\,\dA + C 
    = \Area_{Y_P}(f(W)) + C.
\end{align*}
On the other hand, by Lemma~\ref{lem:increment},
$\Area_{Y_P}(f(W)) \leq \Area_{X_P}(f(W)) + 2\pi|E|$. So
\begin{align*}
    \Area_{X_P}(W)
    \leq \Area_{Y_P}(f(W)) + C 
    &\leq \Area_{X_P}(f(W)) + C + 2\pi|E| \\
    &\leq \Area_{X_P}(W \setminus B) + C + 2\pi|E|.
\end{align*}
Substituting $\Area_{X_P}(W) = \Area_{X_P}(B) + \Area_{X_P}(W \setminus B)$ gives the desired inequality.
\end{proof}

For locally defined functions, the finite invariant sets 
 $P_j$ that we wish to approximate the ``Julia set''
 can no longer necessarily be chosen to be periodic points.
 Instead, we will obtain them by taking preimages 
 of finite subsets of the boundary of their domain
 of definition. 

\begin{lem}\label{lem:sets p_j}
    Let $V \subset \C$ be a bounded open set, and let $f \colon V \to \widehat{\C}$ be meromorphic on a neighbourhood of $\overline{V}$ and non-constant on every component of $V$. Then there exists an increasing sequence of finite subsets $(P_j)_{j=0}^{\infty}$ 
    of $\overline{V}$ such that the following hold. 
    \begin{enumerate}[(1)]
        \item\label{item:atleast2} $|P_j| \geq 2$ for all $j\geq 0$.
        \item\label{item:invariance} $f(P_j \cap V) \subset P_j$ for all
        $j\geq 0$.
        \item\label{item:julia} $J\cap \Omega_V(f)=\varnothing$, where
         $J\defeq \overline{\bigcup_{j=1}^\infty P_j}$.
         \item\label{item:conn_comp} For every connected 
        component $U$ of $\C\setminus J$, 
        either $U$ is a connected component 
        of $\Omega_V(f)$, or $U\cap T_V(f)=\varnothing$.
    \end{enumerate}
\end{lem}

\begin{proof}
    Let $(Q_j)_{j=0}^{\infty}$ be an increasing sequence of finite subsets of $\partial V$ such that $|Q_0| \geq 2$ and $\overline{\bigcup_{j = 1}^\infty Q_j} = \partial V$.
    Define
    $$
        T_{j, 0} \defeq Q_j, \quad T_{j, r + 1} \defeq T_{j, r} \cup (V \cap f^{-1}(T_{j, r})), \quad \text{and}\quad P_j = T_{j, j}.
    $$

    Clearly, $(P_j)$ forms an increasing sequence of finite sets satisfying~\ref{item:atleast2} and~\ref{item:invariance}. 

    Set $P_\infty \defeq \bigcup_{j=1}^\infty P_j$. Since $\bigcup_{j = 1}^\infty Q_j$ is dense in $\partial V$, we have $\partial V \subset J$. Moreover, every point of $P_\infty \cap V$ eventually maps to
    $\partial V$, and hence does not belong to $T_V(f)$. Since $\Omega_V(f)$ is open, it follows that $J \cap \Omega_V(f) = \varnothing$, proving~\ref{item:julia}.

    We next observe that $V \cap f^{-1}(J) \subset J$. Indeed, suppose that
    $z \in V$ and $f(z) \in J$. Let $W \Subset V$ be any neighbourhood of $z$. Then
    $f(W)$ is a neighbourhood of $f(z)$, and hence contains some point $a \in P_\infty$, 
    say $a \in P_j$. Choose $w \in W$ with $f(w)=a$. Since
    $P_j=T_{j,j}$ and $T_{j,j}\subset T_{j+1,j}$, we have
    $w \in T_{j+1,j+1}=P_{j+1}$. Thus every neighbourhood of $z$ intersects $P_\infty$, so $z\in J$.

    Now let $U$ be a component of $\C\setminus J$ such that $U\cap T_V(f)\neq\varnothing$, and choose
    $z\in U\cap T_V(f)$. Since $\partial V\subset J$ and $U$ is connected, we have $U\subset V$.
    The previously proved inclusion $V \cap f^{-1}(J) \subset J$ also implies that $f(U)\cap J=\varnothing$. Since $f(U)$ is connected
    and contains $f(z)\in V$, 
    it is contained in a connected component $U_1$ 
    of $\C\setminus J$ with $U_1\cap T_V(f)\neq\varnothing$. 
    Inductively, we see that $f^n(U)\subset V$
    for every $n\geq0$. Hence $U\subset T_V(f)$, and since $U$ is open, $U\subset\Omega_V(f)$. Moreover,
    $\partial U\subset J$ and $J\cap \Omega_V(f)=\emptyset$,
    so $U$ is indeed a connected component of $\Omega_V(f)$.
   This proves~\ref{item:conn_comp}.   
\end{proof}

\begin{lem}\label{lem:sph diameters and injectivity}
    Let $O\subset\Ch$ be open and $f\colon O\to\Ch$ open and meromorphic. Let $V\Subset O$ be a bounded open subset of $\C$. Let $U$ be a wandering component of $\Omega_V(f)$ such that the connected component $U_n$ of $\Omega_V(f)$ containing $f^n(U)$ is simply connected for every $n \geq 0$.
    For $R > 0$, let $B_{U_n}(z_n,R)$ denote the hyperbolic ball in $U_n$ of radius $R$ centered at $z_n$. Then
    \begin{enumerate}[(1)]
        \item\label{it:sph diameters to zero} the Euclidean diameter of $B_{U_n}(z_n,R)$ tends to $0$ as $n \to \infty$;
        \item\label{it:injectivity} if the closure of the orbit $(z_n)_{n=0}^{\infty}$ is compactly contained in $V$, then there exists $N = N(R)$ such that, for every $n \geq N(R)$, the restriction $f|_{B_{U_n}(z_n,R)}$ is univalent.
    \end{enumerate}
\end{lem}

\begin{proof}

    Let $\pi_n\colon \D \to U_n$ be a Riemann map such that $\pi_n(0) = z_n$. Observe that $B_{U_n}(z_n, R) = \pi_n(B_{\D}(0, r))$ for some $r \in (0,1)$ depending only on $R$.

    Suppose that $\dist(z_n,\partial U_n)$ does not tend to $0$. Then there exist $\delta>0$ and infinitely many $n$ such that $\D_\delta(z_n)\subset U_n$. This is
    impossible since the domains $U_n$ are pairwise disjoint and  the points $z_n$ remain in the fixed bounded set $V$.

    Hence $\dist(z_n,\partial U_n)\to0$. By Koebe's $1/4$-theorem,
    $$
    |\pi_n'(0)| \leq 4\dist(\pi_n(0),\partial U_n)\to0.
    $$
    At the same time, by Koebe's distortion theorem,
    $$
    |\pi_n(w)-\pi_n(0)| \leq \frac{r}{(1-r)^2}|\pi_n'(0)|,
    $$
    for $|w|\leq r$. It follows that the diameter of $B_{U_n}(z_n,R)$ tends to zero.

    Now let us establish~\ref{it:injectivity}. Let $K$ be the closure of the orbit $(z_n)$;
    then $K$ is forward-invariant and, in
    particular, contains no poles of $f$. Choose finitely many sufficiently small Jordan domains $D_i \Subset V$ and Jordan domains $D_i' \Subset D_i$ so
    that the $D_i'$ cover $K$ and
    $$
    f(z) = (\varphi_i \circ g_d \circ \psi_i^{-1})(z) \quad \text{for } z \in D_i,
    $$
    where $g_d(z) = z^d$ for $d \geq 1$, and $\varphi_i\colon \D \to f(D_i) \eqdef G_i$ and $\psi_i\colon \D \to D_i$ are biholomorphisms such that $\varphi_i(0) = f(w_i)$ and $\psi_i(0) = w_i$. In particular, $f$ has no critical points in $D_i$ 
    except possibly $w_i$. 

    Since the diameter of $B_n(R)$ tends to $0$, we may assume that, whenever $z_n \in D_i'$, we have $B_n(R) \Subset D_i$. Also, because
    the  set $E \defeq f(\{z \in V : f'(z) = 0\})$ of critical values of $f$ in $V$
     is finite and our domain
    is wandering, for all sufficiently large $n$, $U_n$, and hence $B_n(R)$, is disjoint from $E$.

    Again, since the diameter of $B_n(R)$ tends to $0$, we may assume that $\overline{B_{n+1}(R)} \subset G_i$, and as above 
    we may assume that 
    $B_{n+1}(R) \cap E = \varnothing$. It then follows from the definition of $G_i$ that each component of $f^{-1}(B_{n+1}(R)) \cap D_i$ is mapped univalently onto $B_{n+1}(R)$. By Pick's theorem, $f(B_n(R)) \subset B_{n+1}(R)$, and hence $B_n(R)$ is contained in the component containing $z_n$. Therefore, $f|_{B_n(R)}$ is univalent, as required.
\end{proof}

Now we are ready to prove Theorem \ref{thm:spherelocal}.

\begin{proof}[Proof of Theorem \ref{thm:spherelocal}\ref{it: no wandering orbit}]

After conjugating $f$ by a M\"obius transformation, we may assume that $V$ is bounded in $\C$. Suppose that the orbit $(z_n)_{n=0}^{\infty}$ is compactly contained in $V$. Choose a compact set $K \subset V$ such that the closure of 
the orbit is compactly contained in $\interior(K)$. We also choose sets $P_j$ as in Lemma~\ref{lem:sets p_j} and distinct points $a,b \in P_1$. Let the finite set $E$ and the constant $C$ be chosen as in Theorem~\ref{thm:localdeficit}.

Fix $R > 0$. By Lemma~\ref{lem:sph diameters and injectivity}, $B_n(R) \defeq B_{U_n}(z_n,R) \subset K$ and $f|_{B_n(R)}$ is injective for all $n \geq N(R)$. Since $f^{-1}(E) \cap V$ is finite, we may also assume that $U_n \cap f^{-1}(E) = \varnothing$ for all $n \geq N(R)$.

By Pick's theorem, we have $f(B_n(R)) \subset B_{n+1}(R)$. Define the union 
\[W_R := \bigcup_{n=N(R)}^\infty B_n(R).\] By
construction, $W_R \subset K \setminus f^{-1}(P_j \cup E)$ for all $j \geq 1$, and $f$ maps $W_R$ injectively into $W_R \setminus B_{N(R)}(R)$.

Lemma~\ref{lem:localcancel}, applied with $P = P_j$, now gives $\Area_{X_j}(B_{N(R)}(R)) \leq C + 2\pi|E|$ for every $j$, where $X_j = \C \setminus P_j$. By Lemma~\ref{lem:convergence of densities}\eqref{it:interior}, $\rho_{X_j}$ converges locally uniformly on $U_{N(R)}$ to~$\rho_{U_{N(R)}}$. Hence
$$
    \Area_{U_{N(R)}}(B_{N(R)}(R)) \leq C + 2\pi|E|,
$$
with a bound independent of $R$. This leads to a contradiction, since
$$
    \Area_{U_{N(R)}}(B_{N(R)}(R)) = 
    \Area_{\DD}(B_{\DD}(0,R))
$$
tends to $\infty$ as $R \to \infty$, since the
area of $\rho_{\DD}$ is infinite. The proof is complete.
\end{proof}

\begin{proof}[Proof of Theorem \ref{thm:spherelocal}\ref{it: no wandering set}]

Again, after conjugating $f$ by a M\"obius transformation, we may assume that $V$ is bounded. Suppose there exists a measurable set $A \subset T_V \setminus \Omega_V$ such that the sets $f^n(A)$ are pairwise disjoint, $f$ is injective on $W \defeq \bigcup_{n=0}^\infty f^n(A)$, and $W$ is compactly contained in $V$. We will show that $A$ has area zero.

Let $K \subset V$ be a compact set such that $W \subset K$. We also choose sets $P_j$ as in Lemma~\ref{lem:sets p_j} and distinct points $a,b \in P_1$. Let the finite set $E$ and the constant $C$ be chosen as in Theorem~\ref{thm:localdeficit}.

Define $S_0 := P_\infty \cup E$, where $P_\infty = \bigcup_{j=1}^\infty P_j$. Inductively, define $S_{r+1} := S_r \cup (V \cap f^{-1}(S_r))$, and let $S_\infty := \bigcup_{r=0}^\infty S_r$. Set $A^* \defeq A \setminus S_\infty$. Recall that
$T_V\setminus \Omega_V\subset J = \overline{P_{\infty}}$ by
Lemma~\ref{lem:sets p_j}~\ref{item:conn_comp}. So
$A^*\subset J\setminus P_{\infty}$. Since $S_\infty$ is countable, removing it from $A$ does not change its area. Let $W^* \defeq \bigcup_{n=0}^\infty f^n(A^*)$. Clearly, $f|_{W^*}$ is injective, $f(W^*) = W^* \setminus A^*$, and $W^* \subset K \setminus f^{-1}(P_j \cup E)$ for all $j \geq 1$.

Lemma~\ref{lem:localcancel}, applied with $P = P_j$, now gives $\Area_{X_j}(A^*) \leq C + 2\pi|E|$ for every $j$, where $X_j = \C \setminus P_j$. It is also easy to see that $A^* \subset J \setminus P_\infty$. Since the densities $\rho_{X_j}(z)$ are increasing by Pick's theorem, the monotone convergence theorem gives
$$
    \lim_{j \to \infty} \Area_{X_j}(A^*) = \int_{A^*} \lim_{j \to \infty} \rho_{X_j}^2 \,\dA \leq C + 2\pi|E|.
$$
However, by Lemma~\ref{lem:convergence of densities}\eqref{it:boundary}, $\rho_{X_j}(z)$ tends to $\infty$ as $j \to \infty$ for every $z \in A^*$. Therefore, $A^*$ must have measure zero.

\end{proof}

\section{Absence of bounded-orbit wandering domains for transcendental entire functions}

\begin{proof}[Proof of Theorem~\ref{thm:boundedorbitentire}]
  Let $f$ be a transcendental entire function with an orbit $(U_n)_{n=0}^{\infty}$ of 
   wandering domains. Suppose, by way of
   contradiction, that the conclusion of the
   theorem does not hold; that is, $U$ is a bounded-orbit wandering domain of $f$. 

  By definition of the Fatou set, $f^n|_U$ is a normal family in the sense of Montel
   (recall that normality is equivalent to equicontinuity by the Arzel\`a--Ascoli theorem). 
   It is well-known that all limit functions of the family of iterates in a 
   wandering domain are constant, and all finite
   such 
   constants belong to the Julia set~\cite[p. 18]{bergweiler93}. The assumption means that 
   the collection of these constants is a compact set $L\subset\C$. 
   
  Let $z_0\in U_0$ and let $D$ be a disc centred at $z_0$ with $\overline{D}\subset  U_0$. 
   Let $V$ be an open Euclidean disc centred
   at the origin, chosen large enough to ensure 
   that $L\subset V$. Then, 
   there is $n_0$ such that $f^n(\overline{D})\subset V$ for $n\geq n_0$. 

 For $n\geq n_0$, let $\tilde{U}_n$ be the connected  component of the set $\Omega_V(f)$ defined 
 in~\eqref{eqn:OmegaV} that contains $f^n(D)$. By the maximum principle,
   each $\tilde{U}_n$ is simply connected, and by Montel's theorem,
   $\tilde{U}_n\subset F(f)$, and hence $\tilde{U}_n\subset U_n$. 

 So $(\tilde{U}_n)_{n=n_0}^{\infty}$ is an orbit of
  simply connected wandering components of $\Omega_V(f)$, and
  hence satisfies the hypotheses of 
  Theorem~\ref{thm:spherelocal}~\ref{it: no wandering orbit}.
  Hence there is a strictly increasing 
  sequence $(m_k)_{k=0}^{\infty}$
  such that $\dist^{\#}(f^{n_0+m_k}(z),\partial V)\to 0$, which is a contradiction to the 
  fact that every limit of the sequence $(f^{n}(z))_{n=0}^{\infty}$ belongs
  to the compact set $L\subset V$.
\end{proof}

\bibliographystyle{amsalpha}

\bibliography{bibliography}
\end{document}